\documentclass[12pt]{amsart}
\usepackage{amsmath}
\usepackage{amsthm}
\usepackage{amscd}
\usepackage[all]{xy} 

\newtheorem{theo}{Theorem}[section]
\newtheorem{prop}[theo]{Proposition}
\newtheorem{lemm}[theo]{Lemma}
\newtheorem{cor}[theo]{Corollary}

\newtheorem{prob}[theo]{Problem}

\theoremstyle{definition}
\newtheorem{defi}[theo]{Definition}
\theoremstyle{remark}

\theoremstyle{definition}

\newcommand{\ddbar}{dd^c}
\newcommand{\reg}{\mathrm{reg}}
\newcommand{\sing}{\mathrm{sing}}

\begin{document}

\title[a converse to the 
Andreotti-Grauert theorem]
{Asymptotic cohomology vanishing and\\ a converse to the 
Andreotti-Grauert theorem\\ on surfaces}

 

\author{SHIN-ICHI MATSUMURA}

\address{Kagoshima University, 1-21-35 Koorimoto, Kagoshima 890-0065, Japan.}

\email{{\tt
shinichi@sci.kagoshima-u.ac.jp\\
mshinichi0@gmail.com}}

\thanks{Classification AMS 2010: 14C20, 14F17, 32L15. }



\maketitle

\begin{abstract}
In this paper, we study relations between positivity of the curvature 
and the asymptotic behavior of the higher cohomology group 
for tensor powers of a holomorphic line bundle.
The Andreotti-Grauert vanishing theorem asserts that 
partial positivity of 
the curvature 
implies asymptotic vanishing of certain higher cohomology groups. 
We investigate the converse implication of this theorem 
under various situations. 
For example, we consider the case where a line bundle is semi-ample or big. 
Moreover, we show the converse implication holds 
on a projective surface  
without any assumptions on a line bundle.
\end{abstract}

\renewcommand\abstractname{R\'ESUM\'E}
\begin{abstract}
Dans cet article, nous \'{e}tudions les relations entre la positivit\'{e}
de la courbure
et le comportement asymptotique de la cohomologie de degré sup\'{e}rieure
des puissances tensorielles d'un fibr\'{e} en droite holomorphe.
Le th\'{e}or\`{e}me d'annulation d' Andreotti-Grauert affirme que
la positivit\'{e} partielle de
la courbure
implique l'annulation asymptotique de la cohomologie de certains degrés sup\'{e}rieures.
Nous \'{e}tudions la réciproque de ce th\'{e}or\`{e}me
dans plusieurs situations.
Par exemple, nous consid\'{e}rons le cas d'un fibr\'{e} en droite
semi-ample ou gros.
De plus, nous montrons que la r\'{e}ciproque du th\'{e}or\`{e}me d' Andreotti-Grauert est vraie
sur les surfaces projectives
sans aucune hypoth\`{e}se sur le fibr\'{e} en droite.
\end{abstract}


\section{Introduction}
In complex geometry, 
the positivity concept plays an important role.
In particular, a positive line bundle is fundamental and important 
in the theory of several complex variables and algebraic geometry.
For this reason, a positive line bundle 
has been characterized in various ways.
For example, some positive multiple gives an embedding to the projective space (geometric characterization), 
all higher cohomology groups 
of some positive multiple are zero (cohomological characterization), 
and the intersection number with any subvariety is positive (numerical characterization). 
The purpose of this paper is to 
generalize these characterizations to a 
{\textit{$q$-positive}} line bundle. 

Throughout this paper, 
let $X$ be a compact complex manifold of dimension $n$,    
$L$ a (holomorphic) line bundle on $X$ 
and $q$ an integer with $0\leq q \leq n-1$.
Sometimes we may assume that $X$ is K\"ahler or projective.

In this paper, we study  
relations between the $q$-positivity and 
the cohomological $q$-amplitude of a line bundle.
The fundamental relations are discussed in \cite{DPS96}. 
K$\mathrm{\ddot{u}}$ronya and Totaro  
investigated 
the cohomological $q$-amplitude of a line bundle 
in terms of algebraic geometry (see \cite{Kur10}, \cite{Tot10}). 
We consider a $q$-ample line bundle in terms of complex geometry. 
Let us recall the definition of a $q$-positive 
(cohomologically $q$-ample) line bundle. 
 
\begin{defi} 
(1) A holomorphic line bundle $L$ on $X$ is called \textit{$q$-positive}, 
if there exists a (smooth) hermitian metric $h$
whose Chern curvature $\sqrt{-1}\Theta_{h}(L)$ 
has at least $(n-q)$ positive eigenvalues at any point on $X$ 
as a $(1,1)$-form.\vspace{1mm}\\
\ \ \ \ (2) A holomorphic line bundle $L$ on $X$ is called \textit{cohomologically $q$-ample}, 
if for any coherent sheaf $\mathcal{F}$ on $X$ 
there exists a positive integer $m_{0}=m_{0}(\mathcal{F})>0$ such that 
\begin{align*}
H^{i}\big( X, \mathcal{F}\otimes \mathcal{O}_{X}(L^{\otimes m})\big) =0\ \ \ \mathrm{for}\  i>q,\ m\geq m_{0}. 
\end{align*}

\end{defi}
A $0$-positive line bundle is a positive line bundle in the usual sense. 
Further, it follows from the Serre vanishing theorem that 
a cohomologically $0$-ample line bundle is an ample line bundle 
in the usual sense of algebraic geometry. 
Thanks to the Kodaira embedding theorem, 
we know that a positive line bundle coincides with an ample line bundle. 
We attend to generalize this relation to a $q$-positive 
line bundle.

Andreotti and Grauert proved that 
a $q$-positive line bundle is always $q$-ample.   
(see \cite[Th\'eor\`eme 14]{AG62}, \cite[Proposition 2.1]{DPS96}).
It is of interest to know whether 
the converse implication of the Andreotti-Grauert theorem holds. 
In this paper, we mainly discuss the following problem. 
\begin{prob}$($\cite[Problem 2.2]{DPS96}$)$. \label{main}
Does the converse implication of the Andreotti-Grauert theorem hold ? 
That is to say, is a $q$-ample line bundle always $q$-positive ?
\end{prob}
This problem was first posed by Demailly, Peternell and Schneider in \cite{DPS96}.
Precisely speaking, they consider a uniformly $q$-ample line bundle. 
However, Totaro showed that 
the uniform $q$-amplitude is the same concept 
as the cohomological $q$-amplitude
(see \cite[Theorem 6.2]{Tot10}). 
When $q$ is zero, this problem is affirmative. 
Therefore it is a natural question. 
However it has been an open problem for a long time, 
except the case of $q=0$.

In Section 2, we study this problem when    
$X$ is a smooth projective surface.
The main result of this section is 
an affirmative answer for Problem \ref{main} on a surface (Theorem \ref{AG1}).

\begin{theo}\label{AG1}
On a smooth projective surface $X$, the converse 
of the Andreotti-Grauert theorem holds.
That is, the following conditions are equivalent.\\
\ \ \ \ $\mathrm{(A)}$ $L$ is cohomologically $q$-ample.
\\
\ \ \ \ $\mathrm{(B)}$ $L$ is $q$-positive.
\end{theo} 
For the proof of Theorem \ref{AG1}, 
we establish Theorem \ref{1-posi}. 
Theorem \ref{1-posi} also leads to Corollary \ref{ps-eff}, which 
can be seen as a generalization of  
\cite[Theorem 1]{FO09}. 
The idea of the proof of 
Theorem \ref{1-posi} is to use a solution of   
the global equation (the Monge-Amp\`{e}re equation).

In his paper \cite{Dem10}, Demailly proved the converse of the holomorphic Morse inequality under various situations.  
These results can be seen as a \lq\lq partial" converse of
the Andreotti-Grauert theorem.
The original part of this paper is to give an \lq\lq exact" converse  
(see \cite{Dem10} and Section \ref{hol} of this paper 
for the precise statement).
By combining Theorem \ref{AG1} and the result of \cite{Dem10}, 
the asymptotic behavior of the higher cohomology 
on a surface
can be interpreted in terms of the curvature.

In Section 3, various characterizations of the $q$-positivity of 
a semi-ample line bundle are given on an arbitrary 
compact complex manifold. 
A line bundle $L$ is called {\textit{semi-ample}}, 
if its holomorphic global sections of some positive multiple of $L$ 
have no common zero set.
A semi-ample line bundle induces a holomorphic map to the projective space.
(See \cite{Laz04} for more details on a semi-ample line bundle.)

Theorem \ref{Main} gives a relation between the fibre dimension of a 
holomorphic map and the $q$-positivity. 
When the map is the holomorphic map  
associated to a sufficiently large multiple of a semi-ample line bundle $L$, 
condition $(B)$ in Theorem \ref{Main} is 
equivalent to the cohomological $q$-amplitude of $L$
(see \cite[Proposition 1.7]{So78}).
It leads to the following theorem:

\begin{theo}\label{Cor1}
Let $L$ be a semi-ample line bundle on 
a compact complex manifold $X$.
Then the following conditions $\mathrm(A)$, $\mathrm(B)$ and $\mathrm(C)$ are equivalent.\\
\ \ \ \ $\mathrm(A)$\ \ $L$ is $q$-positive.\\
\ \ \ \ $\mathrm(B)$\ \ The semi-ample fibration of $L$ has fibre dimensions at most $q$.\\
\ \ \ \ $\mathrm(C)$\ \ $L$ is cohomologically $q$-ample.
\\
Further if $X$ is projective, the conditions above 
are equivalent to condition $\mathrm(D)$.\\
\ \ \ \ $\mathrm(D)$\ \ For every subvariety $Z$
with $\dim{Z}>q$, there exists a curve $C$ on $Z$
such that the degree of $L$ on $C$ is positive.
\end{theo}
Condition (B) (resp. (C), (D)) gives a geometric 
(resp. cohomological, numerical) characterization of 
a $q$-positive line bundle.
In particular, the converse of the Andreotti-Grauert theorem holds 
for a semi-ample line bundle on an arbitrary compact complex manifold.


In section 4, we consider the Zariski-Fujita type theorem 
(Theorem \ref{non-ample}) 
in order to investigate the $q$-positivity of a big line bundle.
In particular, we know that 
the converse of the Andreotti-Grauert theorem 
for a big line bundle 
is reduced to the case of varieties of smaller dimension (the non-ample locus).

\subsection*{Acknowledgment}
The author wishes to express his deep gratitude
to his supervisor Professor Shigeharu Takayama for various 
comments, encouragement and carefully reading a preliminary 
version of this paper.
He  would like to thank Professor Takeo Ohsawa who taught
him the paper \cite{Ba80} which is related with Theorem \ref{Main}.
Further he is deeply grateful to Professor Hajime Tsuji who 
pointed out some mistakes of a preliminary version of this paper. 
He also would like to express his thanks 
to the referee for many nice advices and carefully reading this paper. 
He is supported by the Grant-in-Aid for Scientific Research (KAKENHI No. 23-7228) and the Grant-in-Aid for JSPS
fellows.

\section{The Monge-Amp\`{e}re equation and $(n-1)$-positivity}
This section is devoted to prove Theorem \ref{1-posi} and its corollaries. 
Throughout this section, 
let $L$ be a line bundle on a compact 
K\"ahler manifold $X$ and $\omega $ a K\"ahler form on $X$. 
First we give the proof of Theorem \ref{1-posi}.
\begin{theo}\label{1-posi}
Let $L$ be a line bundle on a compact 
K\"ahler manifold $X$ and 
$\omega$ a K\"ahler form on $X$.
Assume that the intersection number
$\big( L\cdot \{\omega \}^{n-1}\big)$ is positive.
Here $\{\omega \}$ denotes the cohomology class in $H^{1,1}(X, \mathbb{R})$
which is defined by $\omega$.

Then $L$ is $(n-1)$-positive.
That is, there exists a smooth hermitian metric $h$
whose Chern curvature $\sqrt{-1}\Theta_{h}(L)$ 
has at least $1$ positive eigenvalue at every point on $X$.
\end{theo}
\begin{proof}
The main idea of the proof is to use a solution of the Monge-Amp\`{e}re equation.
In order to solve the Monge-Amp\`{e}re equation, 
we make use of the following Calabi-Yau type theorem.
It is a deep result which was proved as a special case in \cite{Yau78}.
Roughly speaking, it says that 
the product of the eigenvalues of a K\"ahler form
(which represents a given K\"ahler class )
can be controlled.

\begin{theo}$($\cite{Yau78}$)$. \label{Yau78}
Let $M$ be a compact  K\"ahler manifold of dimension $n$ and 
$\widetilde \omega $ a  K\"ahler form on $M$.
For a positive smooth $(n,n)$-form $F >0$ with $\int_{M}F= \int_{M} 
\widetilde{\omega}^{n}$, 
there exists a function $\varphi \in C^{\infty}(M, \mathbb{R})$  
with the following properties $:$\\
\ \ \ \ $\mathrm{(1)}$\ \ \ 
$(\widetilde \omega +\ddbar \varphi)^{n} = F$ at every point on $M$\\
\ \ \ \ $\mathrm{(2)}$\ \ \ 
$(\widetilde \omega +\ddbar \varphi)$ is a 
K\"ahler form on $M$.
\end{theo}
Fix a smooth hermitian metric $h$ of $L$.
Then the Chern curvature $\sqrt{-1}\Theta_{h}(L)$
represents the first Chern class of $L$.
We want to construct 
a real-valued smooth function $\varphi $ on $X$ 
such that $\sqrt{-1}\Theta_{h}(L) + \ddbar \varphi $ is 
$(n-1)$-positive (that is, the $(1,1)$-form has at least 
$1$ positive eigenvalue everywhere).
If we obtain a function $\varphi$ with the condition above, 
we can easily see that $L$ is $(n-1)$-positive. 
In fact, the Chern curvature associated to the metric defined by 
$he^{-2\varphi}$ is equal to $\sqrt{-1}\Theta_{h}(L) + \ddbar \varphi$.
Therefore it is sufficient for the proof to construct a 
function $\varphi$ with the condition above.
To construct such function, we use Theorem \ref{Yau78}. 

Since $\omega$ is a positive form, 
the $(1,1)$-form $\sqrt{-1}\Theta_{h}(L) + k \omega $ is 
a K\"ahler form on $X$ for a sufficiently large constant $k>0$.
Now we consider the following Monge-Amp\`{e}re equation:
\begin{align*}
\big( \sqrt{-1}\Theta_{h}(L) + k \omega + \ddbar \varphi \big)^{n}
&= D_{k} \big( k \omega \big)^{n} , \\
\big( \sqrt{-1}\Theta_{h}(L) + k \omega + \ddbar \varphi \big)&>0 .
\end{align*}
Here $D_{k}$ is a positive constant which depends on $k$.
In order to solve this equation, 
we need to define $D_{k}$ by
\begin{align*}
D_{k}:=\dfrac{\int_{X} \big( \sqrt{-1}\Theta_{h}(L) + k \omega \big)^{n}}
{\int_{X} \big( k \omega \big)^{n} }.
\end{align*}
When $D_{k}$ is defined as above, 
we can take a solution of the equation, thanks to Theorem \ref{Yau78}.
In fact, by applying Theorem \ref{Yau78} to
a K\"ahler form $$\widetilde \omega :=\big( \sqrt{-1}\Theta_{h}(L)+ k \omega \big)$$ and 
a smooth $(n,n)$-form $F:=D_{k} \big( k \omega \big)^{n}$, 
we can obtain a solution.
Note that the equality $\int_{X} \widetilde \omega ^{n}= \int_{X}  F $ 
holds by the definition of $D_{k}$.

Now we shall show that 
the constant $D_{k}$ is greater than $1$ for a sufficiently large $k>0$. 
We use the assumption of the theorem only for this argument.
By the definition of $D_{k}$ we have
\begin{align*}
D_{k}&=\dfrac{\int_{X} \big( \sqrt{-1}\Theta_{h}(L) + k \omega \big)^{n}}
{\int_{X} \big( k \omega \big)^{n} }\\
&=\dfrac{\big( L+k\{\omega \} \big)^{n}}
{ k^{n} (\{\omega \}^{n})} \\
&= \dfrac{\big( L^{n} \big) + k 
n \big( L^{n-1}\cdot \{\omega \} \big)
+\dots + k^{n-1} n \big( L \cdot \{\omega \}^{n-1} \big) }
{k^{n} (\{\omega \}^{n})} +1.
\end{align*}

The molecule in the right hand is a polynomial of degree 
$(n-1)$ with respect to $k$. 
Further, the coefficient of the highest degree term is equal to  
$n(L\cdot \{\omega \}^{n-1})$. 
It is positive by the assumption.  
Therefore the first term in the right hand 
is greater than $0$ for a sufficiently large $k>0$.
Hence $D_{k}$ is greater than $1$.

Finally we show that $\sqrt{-1}\Theta_{h}(L) + \ddbar \varphi$ has 
at least $1$ positive eigenvalue at every point on $X$. 
Here $\lambda _{1}(x) \geq \dots \geq \lambda _{n}(x)$ denote 
the eigenvalues of 
$$\big(  \sqrt{-1}\Theta_{h}(L) + k \omega + \ddbar \varphi \big) $$ 
with respect to $k \omega$ at $x \in X$.
Then the function $\lambda _{i}$ for $i=1,2, \dots, n$ is well-defined as a function on $X$.
Further $\lambda _{i}$ for $i=1,2, \dots, n$ is continuous, 
since the $j$-th symmetric function in 
$\lambda _{1}, \dots, \lambda_{n}$ is smooth.
Since $\varphi $ is a solution of the Monge-Amp\`{e}re equation, 
the functions $\lambda _{i}$ satisfy
the following equality everywhere: 
\begin{align*}
\lambda _{1}(x) \lambda _{2}(x) \cdots \lambda _{n}(x) &=D_{k}>1,\ \ \ \ 
{\rm {at\ every point }}\ x\in X.
\end{align*}
In addition, $\lambda _{n}(x)$ is positive for any point $x \in X$
since $$\big(  \sqrt{-1}\Theta_{h}(L) + k \omega + \ddbar \varphi \big) $$ is 
a K\"ahler $(1,1)$-from.
Thus we know $\lambda _{1}(x)>1$ at every point on $X$
since $D_{k}$ is greater than $1$.

Now  
the eigenvalues of 
\begin{equation*}
 \sqrt{-1}\Theta_{h}(L) + \ddbar \varphi = ( \sqrt{-1} \Theta_{h}(L) + k \omega + \ddbar \varphi ) - k \omega
\end{equation*}
are $(\lambda _{1}-1), \dots, (\lambda _{n}-1)$ 
since all eigenvalues of $k \omega $ are $1$. 
Thus $\big(\sqrt{-1} \Theta_{h}(L) + \ddbar \varphi \big)$ has 
$1$ positive
eigenvalue $(\lambda _{1}-1)$ everywhere. 
It completes the proof.  
\end{proof}

On the rest of this section, we give the proofs of 
Theorem \ref{AG} and Corollary \ref{ps-eff}.

\begin{theo}{\rm{(=Theorem \ref{AG1}).}}\label{AG}
On a smooth projective surface $X$, the converse 
of the Andreotti-Grauert theorem holds.
That is, the following conditions are equivalent.\\
\ \ \ \ $\mathrm{(A)}$ $L$ is cohomologically $q$-ample.
\\
\ \ \ \ $\mathrm{(B)}$ $L$ is $q$-positive.
\end{theo} 
In the statement of Theorem \ref{AG}, it follows  
that condition (B) leads to condition (A) 
from the Andreotti-Grauert theorem.
The converse is an open problem and a main subject in this paper.
Theorem \ref{AG} claims Problem \ref{main} 
is affirmatively solved  
on a smooth projective surface.

For the proof of Theorem \ref{AG}, 
we shall prepare Lemma \ref{duality} and Lemma \ref{dual cone}. 
These lemmata may be known facts.
However we give proofs for readers' convenience.
Lemma \ref{duality} can be proved even if $X$ has singularities. 
However for our purpose, 
we need only the case where $X$ is smooth.

\begin{lemm}\label{duality}
Let $L$ be a line bundle on a smooth projective variety $X$.
Then the following conditions are equivalent.
\\
\ \ \ \ $\mathrm{(1)}$ The dual line bundle $L^{\otimes -1} $ 
is not pseudo-effective.
\\
\ \ \ \ $\mathrm{(2)}$ $L$ is cohomologically $(n-1)$-ample.
\end{lemm}
\begin{proof}
This theorem follows from the Serre duality theorem.

First we confirm that condition (2) implies (1).
For a contradiction, we assume that $L^{\otimes -1}$ is psuedo-effective.
Then we can take an ample line bundle $A$ on $X$ such that 
$A\otimes L^{\otimes -m}$ has a non-zero section 
for every positive integer $m>0$. 
Note that $A$ does not depend on $m$. 
For the coherent sheaf $\mathcal{O}_{X}(K_{X}\otimes A^{\otimes -1})$, 
we can take a positive integer $m$ such that 
$$h^{n}\big( X,\mathcal{O}_{X}(
K_{X}\otimes A^{\otimes -1}\otimes L^{\otimes m}) \big)=0$$ 
from condition (2).
Here $K_{X}$ denotes the canonical bundle on $X$.
It follows from the Serre duality theorem that 
$h^{0}\big( X,\mathcal{O}_{X}(
 A\otimes L^{\otimes -m}) \big) =0$ .
This is a contradiction to the choice of $A$.

Conversely we show that condition (1) implies (2). 
Fix an ample line bundle $B$ on $X$. 
For a given coherent sheaf $\mathcal{F}$ on $X$, 
we can take the following resolution of $\mathcal{F}$
by taking a large integer $k>0$:
\begin{equation*}
0\longrightarrow \mathcal{G} \longrightarrow \oplus_{i=1}^{N} \mathcal{O}_{X}(B^{\otimes -k})
\longrightarrow \mathcal{F} \longrightarrow 0.
\end{equation*}
In fact, 
$\mathcal{F}\otimes \mathcal{O}_{X}( B^{\otimes k})$ is globally generated 
for a sufficient large $k$ 
since $B$ is ample. 
Therefore we obtain a surjective map $\oplus_{i=1}^{N} 
\mathcal{O}_{X}(B^{\otimes -k})
\longrightarrow \mathcal{F} $
as a sheaf morphism.
We define $\mathcal{G}$ to be the kernel of its map.

Thus it is sufficient to show that 
there is a positive integer $m_{0}$ such that 
$$h^{n}\big( X, \mathcal{O}_{X}(
B^{\otimes -k}\otimes L^{\otimes m}) \big)=0\ \ {\rm{for}}\ m\geq m_{0}$$
In fact,  
the long exact sequence yields
$h^{n}\big( X, \mathcal{F} \otimes \mathcal{O}_{X}( L^{\otimes m})
 \big)=0$ for $m\geq m_{0}$.
It means that $L$ is cohomologically $(n-1)$-ample.

Since $L^{\otimes -1}$ is not psuedo-effective, 
we can take   
a sufficiently large integer $m_{0}$ such that 
$K_{X}^{\otimes -1}\otimes B^{\otimes k}\otimes L^{\otimes -m}$ is not 
psuedo-effective for $m \geq m_{0}$. 
Since this line bundle is not psuedo-effective, we have  
 $$h^{0}\big( X, \mathcal{O}_{X}(
K_{X}^{\otimes -1}\otimes B^{\otimes k} \otimes L^{\otimes -m}) \big)=0.$$
Again by using the Serre duality theorem, 
we have $$h^{n}(X,\mathcal{O}_{X}( 
B^{\otimes -k}\otimes L^{\otimes m}) )=0$$ for $m \geq m_{0}$.

\end{proof}

\begin{lemm}\label{dual cone}
Let $L$ be a line bundle on a smooth projective variety $X$ 
of dimension $n$.
Then the following conditions are equivalent. \\
\hspace{0.7cm}$\mathrm{(2)}$ $L$ is cohomologically $(n-1)$-ample. \\
\hspace{0.7cm}$\mathrm{(3)}$ There exists a strongly movable curve $C$ on $X$ 
such that the degree of $L$ on $C$ is positive.
\end{lemm}
Here a curve $C$ is called a {\textit{strongly movable}} curve if 
\begin{equation*}
C=\mu _{*} \big( A_{1} \cap \dots \cap A_{n-1} \big)
\end{equation*}
for suitable very ample divisors $A_{i}$ on $\widetilde X$, 
where $\mu :\widetilde X \to X$ is a birational morphism. 
See \cite[Definition 1.3]{BDPP} for more details.

\begin{proof}
The deep result proved in \cite{BDPP} yields Lemma \ref{dual cone}.
It follows from \cite[Theorem2.2]{BDPP} that 
the cone of pseudo-effective line bundles is the dual cone of 
strongly movable curves.
That is, a line bundle is pseudo-effective if and only if 
the degree of the line bundle on every strongly movable curve is semi-positive.
From Lemma \ref{duality}, 
$L$ is cohomological $(n-1)$-ample 
if and only if $L^{\otimes -1}$ is not psuedo-effective.
Therefore then there exists a strongly movable curve $C$ such that
\begin{align*}
(L^{ \otimes -1}\cdot C)<0.
\end{align*}
It completes the proof.
\end{proof}

By applying Lemma \ref{dual cone} and Theorem \ref{1-posi}, 
we shall complete the proof of Theorem \ref{AG}. \vspace{0.3cm}\\
\textit{Proof of Theorem \ref{AG}.}\ \ 
When $X$ is a projective surface, 
the closure of the cone of strongly movable curves coincides with 
the closure of the cone of ample line bundles
(that is, the nef cone).
Indeed, the dual cone of pseudo-effective line bundles equals to
the closure of the cone of ample line bundles.
Therefore, when $L$ is cohomologically $1$-ample, 
there exists an ample line bundle $H$ with $(L\cdot H)>0$ 
by Lemma \ref{dual cone}.

Since $H$ is ample, the first Chern class of $H$ contains  
a K\"ahler form $\omega $.
Since the intersection number $(L\cdot H)$ equals to 
$(L \cdot \{ \omega \})$, the line bundle $L$ satisfies 
the assumption in Theorem \ref{1-posi}.
Therefore if follows that $L$ is $1$-positive from Theorem \ref{1-posi}.
\begin{flushright}
$\square$
\end{flushright}

At the end of this section, 
we prove Corollary \ref{ps-eff}, which can be seen as a generalization of  
\cite[Theorem 1]{FO09} to a psuedo-effective line bundle.
In \cite{FO09}, in order to show 
that an effective line bundle is $(n-1)$-positive, 
Fuse and Ohsawa apply $n$-convexity of a non-compact complex manifold.
We make use of the Monge-Amp\`{e}re equation instead of 
$n$-convexity of a non-compact complex manifold.

\begin{cor} \label{ps-eff}
Let $L$ be a pseudo-effective line bundle on a compact 
K\"ahler manifold $X$.
Assume that the first Chern class of $L$ is not trivial.

Then $L$ is $(n-1)$-positive.
\end{cor}
A line bundle is called {\textit{pseudo-effective}} if 
there exists a singular hermitian metric $h$
whose curvature current $\sqrt{-1}\Theta_{h}(L)$ 
is positive on $X$ as a $(1,1)$-current.
A pseudo-effective line bundle (which is not numerically trivial) 
is cohomologically $(n-1)$-ample (see Lemma \ref{duality}). 
Therefore a pseudo-effective line bundle should 
be $(n-1)$-positive if the converse of 
the Andreotti-Grauert theorem holds.
Corollary \ref{ps-eff} claims that it is affirmative 
on a compact K\"ahler manifold.
\begin{proof}
Under the assumption of Corollary \ref{ps-eff}, 
we show that the line bundle 
$L$ satisfies the assumption in Theorem \ref{1-posi}.

Fix a K\"ahler form $\omega$ on $X$ and  
take an arbitrary smooth $(n-1,n-1)$-form $\gamma $ on $X$.
The $(n-1,n-1)$-form $\omega ^{n-1}$ is strictly positive.
Therefore there exists a large constant $C>0$ such that 
$$-C  \omega ^{n-1} \leq \gamma \leq C  \omega ^{n-1}.$$
Here we used the compactness of $X$. 
Since $L$ is psuedo-effective, 
there exists a singular hermitian metric such that 
$\sqrt{-1}\Theta_{h}(L)$ is a positive current.
Since $\sqrt{-1}\Theta_{h}(L)$ is a positive current, 
the following inequalities hold:
\begin{align*}
-C \int_{X} \sqrt{-1}\Theta_{h}(L) \wedge \omega ^{n-1}
\leq \int_{X} \sqrt{-1}\Theta_{h}(L) \wedge \gamma 
\leq C \int_{X} \sqrt{-1}\Theta_{h}(L)  \wedge \omega ^{n-1}.
\end{align*} 
The positive current 
$\sqrt{-1}\Theta_{h}(L)$ represents the first Chern class  
of $L$.
Thus the integral $\int_{X} \sqrt{-1}\Theta_{h}(L) \wedge \omega ^{n-1}$ agrees with 
the intersection number $
\big( L \cdot \{\omega \}^{n-1} \big)$.
If the intersection number is zero, it follows from the inequality above that 
$\int_{X} \sqrt{-1}\Theta_{h}(L)  \wedge \gamma$ is zero
for any smooth $(n-1,n-1)$-form $\gamma $.
It means that $ \sqrt{-1}\Theta_{h}(L) $ is a zero current.
This is a contradiction to the assumption that 
the first Chern class of $L$ is not trivial.
Hence the intersection number $\big( L \cdot \{\omega \}^{n-1} \big)$ must 
be positive.
It follows that $L$ is $(n-1)$-positive from Theorem \ref{1-posi}.
\end{proof}

\section{The fiber dimension and $q$-positivity}
The main purpose of this section is to give the proof of Theorem \ref{Cor1}. 
For this purpose, we first consider Theorem \ref{Main}. 

\begin{theo}\label{Main}
Let $\Phi : X \longrightarrow Y$ be a holomorphic map $($possibly not surjective$)$ 
from $X$ to a compact complex manifold $Y$.
Then the following conditions are equivalent. 
\vspace{0.2cm}\\
\ \ \ \ $\mathrm{(A)}$\ \ Fix a Hermitian form $\omega $ $($that is, a positive definite $(1,1)$-form$)$ on $Y$.
Then there exists a  
function $\varphi \in C^{\infty}(X, \mathbb{R})$ 
such that the $(1,1)$-form $\Phi ^{*}\omega + \ddbar \varphi $ is $q$-positive 
$\mathrm{(}$that is, the form has at least $(n-q)$ positive eigenvalues at any point on $X$ as a $(1,1)$-form$\mathrm{)}$.
\vspace{0.1cm}\\
\ \ \ \ $\mathrm{(B)}$\ \ The map $\Phi $ has fibre dimensions at most $q$.
\end{theo}

Throughout this section, 
let $\Phi : X \longrightarrow Y$ be a holomorphic map  
from $X$ to a compact complex manifold $Y$.
Fix a hermitian form $\omega $ on $Y$.
Set $\widetilde { \omega }:=\Phi ^{*} \omega$, 
which is a semi-positive $(1,1)$-form on $X$.

First we show that condition (A) implies (B). 
For a contradiction, we assume that 
there is a fibre $F$ of the map $\Phi$ with $\dim{F}>q$. 
Then by condition (A), $X$ 
allows a smooth function $\varphi $ 
such that $\widetilde { \omega } + \ddbar \varphi $ is $q$-positive. 
Since $F$ is a fibre, the restriction to $F$ of 
$\widetilde { \omega }=\Phi ^{*} \omega$ is equal to zero. 
It implies that the restriction to $F$ of $\ddbar \varphi$ 
is $q$-positive. 
Even if $F$ has singularities, we can define the $q$-positivity 
(see Definition \ref{sing}).  
Since the dimension of $F$ is strictly larger than $q$, 
the Levi-form of $\varphi |_{F}$ 
has at least $1$ positive eigenvalue.

Since $F$ is compact, the function $\varphi |_{F}$ 
must have the maximum value on $F$. 
Suppose that $p \in F$ attains the maximum value of $\varphi |_{F}$. 
Then the Levi-form of $\varphi |_{F}$ at $p$ has no positive eigenvalues. 
It is a contradiction. 
Hence condition (A) leads to (B). 
\vspace{0.2cm}

On the rest of this section, 
we shall show that condition (B) implies (A). 
From now on, we assume that 
the dimension of every fibre of the map $\Phi $ is at most $q$. 
Then we want to construct a function $\varphi$ with condition (A). 
For this purpose, 
we define the degenerated locus of a hermitian form by the pull-back of 
the map $\Phi $ as follows:
\begin{defi}
The \textit{degenerated locus}  
by the pull-back of $\Phi$ is defined to be 
\begin{equation*}
B_{q} := \{ \ p \in X \ \big|\ \Phi ^{*} \omega \ {\rm has\ at\ least}
\ (q+1)\ {\rm zero}
{\ \rm eigenvalues\ at}\ p  \}.
\end{equation*}
\end{defi}
Since $\widetilde { \omega }=\Phi ^{*} \omega$ is a semi-positive form, $\widetilde { \omega }$ is $q$-positive outside $B_{q}$.
Therefore 
if $B_{q}$ is empty, condition $(A)$ is satisfied by taking $\varphi  :=0$.
Thus it is sufficient to consider the case where $B_{q}$ is not empty.
The following lemma asserts that the degenerated locus is locally 
the zero set of finite holomorphic functions.
\begin{prop} \label{subvariety}
The degenerated locus $B_{q}$ is a closed analytic set on $X$. 
\end{prop}

\begin{proof}
First we show that $B_{q}$ is a closed set in $X$.
Fix a hermitian form $\overline{\omega} $ on $X$.
We denote by $\lambda _{1} \geq \dots \geq \lambda _{n}\geq 0$,  
the eigenvalues of $ \widetilde { \omega }$ with 
respect to $\overline{\omega}$.
The $j$-th symmetric function in 
$\lambda _{1}, \dots ,\lambda_{n}$ are smooth since 
$\widetilde { \omega }$ and $\overline{\omega}$ is smooth forms.
Therefore 
$\lambda_{i} $ for $0 \leq i \leq n $ is a (well-defined)
continuous function on $X$. 
Now $p$ is contained in $B_{q}$ if and only if $\lambda _{n-q}(p)=0$.
Thus the degenerated locus is closed since $\lambda _{n-q}$ is a 
continuous function.

It remains to show that $B_{q}$ is the zero set of finite holomorphic functions.
Now we take a local coordinate $(w_{1}, \dots, w_{m})$ on $Y$.
Here $m$ denotes the dimension of $Y$.
Then the degenerate locus of $\omega $ coincides with the degenerate locus of $\sum_{i=1}^{m} \ddbar |w_{i}|^{2}$.
In fact, it follows since 
we have 
$$({1}/{C}) \sum_{i=1}^{m} \ddbar |w_{i}|^{2} \leq \omega  \leq  C \sum_{i=1}^{m} \ddbar |w_{i}|^{2} $$ 
for a sufficiently large constant $C>0$.

The holomorphic map $\Phi $ can be locally written as 
$(z_{1}, \dots, z_{n}) \longmapsto 
(f_{1}(z), \dots, f_{m}(z))$
for some holomorphic functions $f_{i}$. 
Here $(z_{1}, \dots, z_{n})$ denotes a local coordinate on $X$. 
Then $B_{q}$ is equal to the locus where the
hermitian form $\ddbar \sum_{i=1}^{m} \ddbar |f_{i}(z)|^{2}$
has at least $(q+1)$ zero eigenvalues.
In general, a semi-positive hermitian form has at least 
$(q+1)$ zero eigenvalues if and only if 
$j$-th symmetric function in the eigenvalues is zero for $n-q \leq j \leq n$.
By a simple computation, $j$-th symmetric function $\sigma _{j}$ in 
$\lambda _{1}, \dots ,\lambda_{n}$
is described as follows:
\begin{align*}
\sigma _{j} &= \sum_{0\leq  i_{1}< \dots < i_{j} \leq n} \det \Big( \big< \mathbb{G}_{i_{a}}, \mathbb{G}_{i_{b}} \big> \Big) _{a,b=1, \dots, n}\\
&=\Big| \sum_{0\leq  i_{1}< \dots <i_{j} \leq n} \sum_{0\leq  \alpha _{1}< \dots <\alpha _{j} \leq m}
\det \Big( g^{\alpha_{a}}_{i_{b}} \Big) _{a,b=1, \dots, j} \Big|^{2}. 
\end{align*}
Here $g^{\alpha }_{i}$ denotes the differential ${\partial f_{\alpha
}}/{\partial z_{i}}$,  
$\mathbb{G}_{i}$ a vector $(g^{1}_{i}, \dots, g^{m}_{i})$, 
and $\big< \cdot, \cdot \big>$ a standard hermitian metric on $\mathbb{C}^{m}$.
Therefore the set defined by $\sigma _{j}=0$ $(n-q \leq j \leq n)$ is 
the zero set of finite holomorphic functions.
\end{proof}

Thanks to Proposition \ref{subvariety}, we can consider the dimension of $B_{q}$.
When the dimension of $B_{q}$ is less than or equal to $q$, 
we can easily see condition (A) 
in Theorem \ref{Main} from Lemma \ref{potential2}.
When the dimension is greater than $q$, 
we factor $B_{q}$ to subvarieties of smaller dimension.  
For this purpose we need Lemma \ref{factor}.
The assumption on fibre dimensions is used in the proof of this lemma.
Later we need to treat an analytic set 
which may not be closed. 
For that reason, Lemma \ref{factor} is formulated
for an analytic set (possibly not closed).
\begin{defi}
A subset $V$ in $X$ is called {\textit {an analytic set}}, if 
for every point $p$ in $V$ 
there exist a small neighborhood of $p$ and finite holomorphic functions on 
the neighborhood 
such that $V$ is the zero set of these functions. 
\end{defi}

Note that an analytic set does \lq \lq not" mean a closed analytic set in this paper. 
For example, the set 
$\big\{ {1}/{n} \in \mathbb{C}\ \big|\
 n \in\mathbb{N} \big\} $ is an analytic set, but 
not a closed analytic set.

\begin{lemm}\label{factor}
Let $W$ be an 
irreducible analytic set $($possibly not closed, singular$)$ on $X$.
Assume that the dimension of $W$ is greater than $q$.

Then the degenerate locus $D$ defined by 
\begin{equation*}
D:= \big\{p \in W_{\reg}\ \big| \ 
{\rm The\ restriction}\ \widetilde { \omega }|_{W_{\reg}}\ 
{\rm has\ at\ least}\ (q+1)\ {\rm zero}\ {\rm eigenvalues\ at}\ p \ \big\}
\end{equation*}
is a closed analytic set on $W_{\reg}$ and properly contained in $W_{\reg}$.
Here $W_{\reg}$ denotes the regular locus of $W$. 
\end{lemm}

\begin{proof}
We have already proved that 
$D$ is a closed analytic set on $W_{\reg}$ 
in the proof of Proposition \ref{subvariety}.
It remains to show that $D$ is a properly contained subset in $W_{\reg}$.
For a contradiction, we assume $D=W_{\reg}$.

We take a point $p$ in $W_{\reg}$
such that 
$\Phi|_{W}:W \longrightarrow Y $ is a smooth morphism at $p$, and 
a fibre $F$ of $\Phi$ containing $p$.
Further we take an open ball $U$ in $W$ with 
a local coordinate $(z_{1}, \dots, z_{r})$ centered at $p$. 
We may assume that 
the first coordinate $(z_{1}, \dots,. z_{s})$ also becomes a local 
coordinate on $F_{\reg}$. 
Here $r$ (resp. $s$) denotes 
the dimension of $W$ (resp. $F$). 

Now  
we consider the restriction $f$ of $\Phi$ 
defined by 
\begin{align*}
 f:= \Phi|_{F^{\perp }} : F^{\perp} \longrightarrow Y,\ 
{\rm where}\ F^{\perp}=\{ \ (0, \dots, 0, z^{s+1}, \dots, z^{r} ) \in U\ \}.
\end{align*}
Then the fibre of $\Phi (p)$ by $f$ is a zero dimensional analytic set.
It implies the Jacobian of $f$ is not identically zero on 
some neighborhood of $p$.
Hence the holomorphic map $f$ gives a local 
biholomorphic at some point.
This means the restriction of $\widetilde { \omega }$ to $F^{\perp}$ has $(n-s)$ positive eigenvalues.
Note that $s$ is less than or equal to $q$.
This is a contradiction to $W_{\reg}=D$.
\end{proof}

Lemma \ref{factor} leads to Proposition \ref{Key}.
Later we shall apply this proposition to $B_{q}$.
The set $B_{q}$ is a closed analytic set.
However we formulate this proposition for an analytic set 
to prove Proposition \ref{Key} with induction on the dimension.
Remark that \lq\lq dimension " in Proposition \ref{Key} does not necessarily
mean the pure dimension .
\begin{prop}\label{Key}
Let $V$ be an analytic set of dimension $k$ $($possibly, not closed, not 
irreducible, singular$)$.
Then there exist sets $D_{\ell}$ $(0\leq \ell \leq k-1)$ 
with following properties: 

$\mathrm{(1)}$\ \ $D_{\ell}$ is an analytic set on $X$.

$\mathrm(2)$\ \ $D_{k}:=V \supseteq D_{k-1}\supseteq \dots 
\supseteq D_{1} \supseteq D_{0}$.

$\mathrm(3)$\ \  $\dim{D_{\ell}}=\ell$ for $\ell=0, 1, 2, \dots, k$.

$(4)$\ \ $D_{\ell} \setminus D_{\ell-1}$ for $\ell=1, 2, \dots, k$ is a disjoint union of 
non-singular analytic sets.

$\mathrm(5)$ For an irreducible component $W$ of 
$D_{\ell} \setminus D_{\ell-1}$ with $\dim{W}>q$, 
the $(1,1)$-form $\widetilde { \omega }|_{W}$ has 
$(\dim{W}-q)$ positive eigenvalues.
\end{prop}

\begin{proof}
We prove this proposition 
by induction on the dimension $k=\dim{V}$.
When $k$ is zero, we set $D_{0}=V$.
Then the properties in Proposition \ref{Key} hold.
From now on, 
we assume that $k$ is greater than zero.

We consider the decomposition $V=V_{\reg} \cup V_{\sing}$.
Here $V_{\reg}$ (resp. $ V_{\sing}$) denotes 
the regular locus (resp. the singular locus) of $V$.
Note that this decomposition is a disjoint union.
Since the dimension of $V_{\sing}$ is smaller than $k$, we obtain $\widetilde D_{\ell}$ 
\ ($0 \leq \ell \leq \dim{V_{\sing}}$) with the properties 
in Proposition \ref{Key} by the induction hypothesis.

Let $V_{\reg}=\bigcup _{i \in I} V_{i}$ be
the irreducible decomposition of $V_{\reg}$.
For an irreducible component $V_{i}$, we set $D^{i}_{\dim{V_{i}}}:=V_{i}$ if the dimension of $V_{i}$ is less than or equal to $q$.
Otherwise, we investigate the degenerated locus $D^{i}$ of $V_{i}$ defined by 
\begin{equation*}
D^{i}:= \{p \in {(V_{i})}_{\reg} \ \big| \ 
\widetilde { \omega }|_{{(V_{i})}_{\reg}}\ 
{\rm has\ at\ least}\ (q+1)\ {\rm zero}\ {\rm eigenvalues\ at}\ p\  \}.
\end{equation*}
It follows that $D^{i}$ is an analytic set and 
properly contained in $V_{i}$ from Lemma \ref{factor}. 
In particular, the dimension of $D^{i}$ is smaller than $k$.  
Therefore by applying the induction hypothesis to $D^{i}$, 
we obtain $D^{i}_{\ell}$ ($0 \leq \ell \leq \dim{V_{i}}$)
with the properties in Proposition \ref{Key}.

For each $\ell$, we define the set $D_{\ell}$ to be the union of $\widetilde D_{\ell}$ and $D^{i}_{\ell}$ ($i \in I$). 
Then we can easily see that 
$D_{\ell}$ has the properties in Proposition \ref{Key}. 
In fact, it follows since  
$\widetilde D_{\ell}$ and $D^{i}_{\ell}$ 
satisfy the properties in Proposition \ref{Key} and 
$D_{\ell}$ is a disjoint union of them.
\end{proof}

For the proof of Theorem \ref{Main}, 
we apply Proposition \ref{Key} to $B_{q}$. 
Then we obtain $D_{\ell}$ with the properties in Proposition \ref{Key}.
By using these properties, 
we shall construct a function $\varphi$
whose Levi-form has $(n-\dim{W})$ positive eigenvalues in the normal direction of 
an irreducible component $W$ of $D_{\ell} \setminus D_{\ell-1}$.
On the other hand, the restriction $\widetilde{\omega}|_{W}$ has $(\dim{W}-q)$ positive eigenvalues from property (5)
if the dimension of $W$ is greater than $q$.
Thus if there is such function $\varphi$, 
the $(1,1)$-form $\widetilde{\omega} + \ddbar \varphi $ 
has at least $(n-q)$ positive eigenvalues everywhere.
To construct such function, 
we prepare Proposition \ref{potential}, \ref{normal}. 
The proofs of these Propositions are similar to the proof  
of \cite[Theorem 4]{Dem90} (cf. \cite{Siu77}). 

\begin{prop}\label{potential}
For $\ell=0,1,\dots, k$, there exists a function $\varphi _{\ell}
\in C^{2}(X, \mathbb{R})$ 
with the following properties $:$

$\mathrm(1)$\ \ Let $W$ be 
an irreducible component of $D_{\ell} \setminus D_{\ell-1}$. 
Then the Levi form $\ddbar \varphi _{\ell}$ has $(n-\dim{W})$ positive eigenvalues in the normal direction of $W$.

$\mathrm(2)$\ \ The Levi form $\ddbar \varphi _{\ell}$ is semi-positive at every point in $\overline {D_{\ell}}$.
\end{prop}
Here $\overline {D_{\ell}}$ denotes the closure of $D_{\ell}$ in $X$.
Note $\overline {D_{\ell}}$ may not be a closed analytic set. 
For example, the closure of $\{(x,y)\in \mathbb{C}^{2}\ \big|\ x=e^{y}\}$ in the 2-dimensional projective space is not 
a closed analytic set.
In order to show Proposition \ref{potential}, 
we prepare the following lemma.
If $D_{\ell}$ for $\ell=0,1,\dots, k$ is a closed analytic set, 
the statement of this lemma is same as that of Proposition \ref{potential}.
\begin{lemm}\label{potential2}
Let $\ell$ be an integer with $0\leq \ell \leq k$.
For every open neighborhood $U$ of $\overline{ D_{\ell} } \setminus D_{\ell}$, 
there exists a function $\varphi _{U} \in C^{\infty}(X, \mathbb{R})$
with following properties $:$

$\mathrm(1)$\ \ Let $W$ be an irreducible component of 
$D_{\ell} \setminus D_{\ell-1}$.
Then the Levi form $\ddbar \varphi _{U}$ has $(n-\dim{W})$ positive eigenvalues in the normal direction at every point in $W\setminus U$.

$\mathrm(2)$\ \ The Levi form $\ddbar \varphi _{U}$ is semi-positive at every point in $\overline {D_{\ell}}$.
\end{lemm}
\begin{proof}
For a given $U$, we take an open covering of $X$ by open balls $U_{j}$ $(j=1,2,\dots, N)$.
Further, we can assume the following properties
for this covering. \vspace{0.1cm}

$\mathrm(a)$\ \ The members $U_{j}$ $(j=1,2,\dots, s)$ cover $D_{\ell} \setminus D_{\ell-1}$.

$\mathrm(b)$\ \ The members $U_{j}$ $(j=1,2,\dots, s)$ are 
contained in $U$. 

$\mathrm(c)$\ \ Every $U_{k}$ which intersects with $U_{j}$ $(j=1,2,\dots, s)$ is also contained in $U$.
\vspace{0.2cm}

We denote by $\mathcal{I}_{B_{q}}$, the ideal sheaf associated to 
a closed analytic set $B_{q}$. 
For every $j=1,2,\dots, s$, we take 
holomorphic functions 
$\{ f_{j,i} \}_{i=1}^{N_{j}}$ on $U_{j}$ such that  
these functions generate the ideal sheaf $\mathcal{I}_{B_{q}}$.
Further, for every $j=s+1,\dots, N$, we 
take holomorphic functions 
$\{ f_{j,i} \}_{i=1}^{N_{j}}$ on $U_{j}$ such that 
these functions generate the ideal sheaf $\mathcal{I}_{D_{\ell}}$. 
Note $D_{\ell}$ is a closed analytic set on $U_{j}$ $(j=s+1,\dots, N)$. 
Therefore we can define the ideal sheaf $\mathcal{I}_{D_{\ell}}$ and 
take its generators.

Let $\big\{ \rho _{j} \big\}_{j=1}^{N}$ be a partition of unity 
associated to the covering of $X$.
Now we define $\varphi _{U}$ to be 
\begin{equation*} 
\varphi _{U}:=\displaystyle \sum_{j=1}^{N} \rho _{j} 
\bigg( \sum_{i=1}^{N_{j}} |f_{j,i}|^{2} \bigg) .
\end{equation*}

Then $\varphi _{U}$ satisfies properties (1), (2).
In fact, an easy computation yields 
\begin{align*}
\sqrt{-1}\partial \overline{\partial} \varphi_{U}= \sum_{j=1}^{N}  
 \sum_{i=1}^{N_{j}}
\Big {\{} 
|f_{j,i}|^{2} \sqrt{-1}\partial \overline{\partial} \rho _{j}
&+ \sqrt{-1} \overline {f_{j,i}} \partial f_{j,i} \wedge \overline {\partial} \rho _{j} \\
&+ \sqrt{-1}{f_{j,i}} \partial \rho _{j} \wedge \overline{\partial} f_{j,i}
+ \sqrt{-1} \rho _{j} \partial f_{j,i} \wedge \overline{\partial} \overline{f_{j,i}} 
\Big {\}}.
\end{align*}
By the definition, $f_{j,i}$ is identically zero on $\overline {D_{\ell}}$.
(Notice that $\overline {D_{\ell}}$ is contained in $B_{q}$.)
Therefore the first three terms are 
zero on $\overline {D_{\ell}}$.
Further the last term is clearly semi-positive.
Therefore property (2) is satisfied.
For every point $p$ in $W\setminus U$, we can take $j_{0}$
such that $U_{j_{0}}$ does not intersect with $\overline{ D_{\ell} } \setminus D_{\ell-1}$ 
and $\rho _{j_{0}}(p) \not = 0$ by the choice of the covering.
Hence the last term has property (2)
since $\partial f_{j_{0},i} \wedge \overline{\partial} \overline{f_{j_{0},i}}$ 
has $(n-\dim{W})$ positive eigenvalues at $p$ in the normal direction 
of $W$.
\end{proof}

Before the proof of Proposition \ref{potential}, 
we recall the definition of a $C^{2}$-norm $\parallel \cdot \parallel_{C^{2}}$ 
on $C^{2}(X, \mathbb{R})$. 
We take an open covering of $X$ by open balls
$U_{j}$ $(j=,1, 2, \dots, N)$ with 
a differential coordinate $(x_{1}^{j}, \dots, x_{2n}^{j})$.
Let $V_{j}$ be a relatively compact set in $U_{j}$ such that 
$\{ V_{j} \}_{j=1}^{N}$ is also an open covering of $X$.
Then the $C^{2}$-norm 
$| \cdot |_{C^{2}}$ 
with respect to the open covering 
is defined to be  
\begin{align*}
| f |_{C^{2}}:= \sum_{j=1}^{N} \sum
_{\alpha, \beta =1 }^{2n, 2n}
\ \sup_{
p \in \overline{V_{j}}}\ \bigg|  
\frac{\partial^{2} f}{\partial x^{j}_{\alpha }\partial x^{j}_{\beta  }}(p) \bigg|
+\sum_{j=1}^{N} \sum
_{\alpha =1 }^{ k}
\ \sup_{
p \in \overline{V_{j}}}\ \bigg|  
\frac{\partial f}{ \partial x^{j}_{\alpha } }(p) \bigg| 
+ \sup _{p \in X} \big| f(p) \big|
\end{align*}
for every $f \in C^{2}(X, \mathbb{R})$. 
Certainly the norm depends on the choice of an open covering. 
However the topology induced by these norms is unique. 
For our purpose, we fix the norm. 
Let us begin the proof of Proposition \ref{potential}. 
\vspace{0.3cm}\\
{\it Proof of Proposition \ref{potential}.}
\ \ Choose a family of open neighborhoods 
$\{ U_{i}\}_{i=1}^{\infty}$ 
of $\overline{ D_{\ell} } \setminus D_{\ell-1}$ such that 
$\bigcap _{i =1}^{\infty}U_{i}= \overline{ D_{\ell} } \setminus D_{\ell-1}$.
For each positive integer $i$, we can take a function $\varphi _{U_{i}} \in C^{\infty}(X, \mathbb{R})$ with 
the properties in Lemma \ref{potential2}.
We set
\begin{equation*}
A_{i}=\parallel \varphi _{U_{i}}\parallel_{C^{2}}.
\end{equation*}
Here $\parallel \cdot \parallel_{C^{2}}$ denotes the fixed $C^{2}$-norm.

Now we define a function $\varphi _{\ell}$ on $X$ 
to be 
\begin{equation*}
\varphi _{\ell}:= \sum_{i =1}^{\infty} 
\dfrac{1}{2^{i}A_{i}}\varphi _{U_{i}}.
\end{equation*}
By the definition of $A_{i}$, 
the sum in the definition uniformly converges with respect to 
the $C^{2}$-norm.
Hence we obtain
\begin{equation*}
\ddbar \varphi _{\ell}= \sum_{i =1} ^{\infty}
\dfrac{1}{2^{i}A_{i}}\ddbar \varphi _{U_{i}}.
\end{equation*}
Properties (1), (2) in Lemma \ref{potential2} and 
the choice of $U_{i}$ lead to 
the properties in Proposition \ref{potential}.
\vspace{-1mm}
\begin{flushright}
$\square$
\end{flushright}

\begin{prop}\label{normal}
For every integer $\ell=0,1,\dots, k$, 
there exists a function $\widetilde {\varphi _{\ell}} \in 
C^{2}(X, \mathbb{R})$
with property $\mathrm{(*)}$.\vspace{0.1cm}\\
\ \ \ \ $\mathrm{(*)}$\ \ 
Let $m$ be an integer with $0 \leq m \leq \ell$ and 
$W$ an irreducible 
component of $D_{m} \setminus D_{m-1}$.
Then the Levi-form $\ddbar \widetilde {\varphi _{\ell}}$ has $(n-\dim{W})$ positive eigenvalues in the normal direction of $W$.

\end{prop}

Before the proof of Proposition \ref{normal}, 
we confirm that
Proposition \ref{normal} and 
Proposition \ref{Key} complete the proof of Theorem \ref{Main}.
That is, there is a smooth function $\varphi $ on $X$
such that the $(1,1)$-form $\widetilde{\omega} + \ddbar \varphi $ 
is $q$-positive.

First we obtain $\{ D_{\ell} \}_{\ell=0}^{k}$ with properties in 
Proposition \ref{Key}
by applying Proposition \ref{Key} to $B_{q}$. 
Now we take 
$\widetilde {\varphi _{k}}$ with property $(*)$ in Proposition \ref{normal}.
Recall $k$ is the dimension of $B_{q}$. 
Then we shall show that 
$\widetilde{\omega} + \varepsilon \ddbar  \widetilde {\varphi} _{k} $
is $q$-positive for a sufficiently small $\varepsilon >0 $.

Now $\widetilde{\omega}$ is $q$-positive at $x \not \in B_{q}$.
Hence when $x$ is not contained in  $B_{q}$, 
the form $\widetilde{\omega} + \varepsilon \ddbar  \widetilde {\varphi} _{k} $ 
is $q$-positive for a small $\varepsilon >0$. 
When $x$ is contained in $B_{q}$, 
there is a positive integer $\ell$ such that 
$x \in D_{\ell}\setminus D_{\ell-1}$.
(Otherwise $x$ is contained in $D_{0}$.
Then the Levi-form of 
$\widetilde {\varphi} _{k}$ has $n$ positive eigenvalues at $x$.)
For an irreducible component $W$ of $D_{\ell}\setminus D_{\ell-1}$ 
containing $x$, 
the $(1,1)$-form $\widetilde{\omega}|_{W}$ has $(\dim{W}-q)$ positive eigenvalues 
along $W$.
On the other hand, the Levi-form of $\widetilde {\varphi} _{k}$ has 
$(n-\dim{W})$ positive eigenvalues in 
the normal direction of $W$ (property $(*)$ in Proposition \ref{normal}).
Thus $\widetilde{\omega} + \varepsilon \ddbar 
\widetilde {\varphi} _{k}$ has $(n-q)$ positive eigenvalues.
The function $\widetilde {\varphi} _{k}$ may not be smooth.
However we can approximate it with smooth functions without 
the loss of the $q$-positivity since it is $C^{2}$-function 
(for instance, see \cite{Ric68}). 
It completes the proof of Theorem \ref{Main}.
\vspace{2mm}
\\
{\it Proof of Proposition \ref{normal}.}\ \ 
We prove Proposition \ref{normal} by induction on $\ell$.
When $\ell$ is zero, the claim is obvious.
Thus we assume $\ell$ is greater than $0$.
By the induction hypothesis, we obtain a smooth function 
$\widetilde {\varphi} _{\ell-1}$ with property $\mathrm {(*)}$.
Further we take $\varphi _{\ell}$ with the properties 
in Proposition \ref{potential}.
We define a function $\widetilde {\varphi} _{\ell}$ to be 
$\varphi _{\ell}+ \varepsilon   \widetilde {\varphi} _{\ell-1}$.
Then the function satisfies property $\mathrm {(*)}$ for a sufficiently small 
$\varepsilon > 0$.
\begin{flushright}
$\square$
\end{flushright}\vspace{-5mm}\ 
\\
Finally we see that it follows Theorem \ref{Cor1} from Theorem \ref{Main}.
\begin{theo}$(${\rm{=Theorem \ref{Cor1}}}$)$.
Let $L$ be a semi-ample line bundle on 
a compact complex manifold $X$.
Then the following conditions $\mathrm(A)$, $\mathrm(B)$ and $\mathrm(C)$ are equivalent.\\
\ \ \ \ $\mathrm(A)$\ \ $L$ is $q$-positive.\\
\ \ \ \ $\mathrm(B)$\ \ The semi-ample fibration of $L$ has fibre dimensions at most $q$.\\
\ \ \ \ $\mathrm(C)$\ \ $L$ is cohomologically $q$-ample.
\\
Further if $X$ is projective, the conditions above 
are equivalent to condition $\mathrm(D)$.\\
\ \ \ \ $\mathrm(D)$\ \ For every subvariety $Z$
with $\dim{Z}>q$, there exists a curve $C$ on $Z$
such that the degree of $L$ on $C$ is positive.
\end{theo}

\begin{proof}
We can easily confirm  
the equivalence between condition (B) and (C)
from the standard argument of the spectral sequence 
(see \cite[Proposition 1.7]{So78}).

The equivalence between condition (A) and (B) is directly 
followed from Theorem \ref{Main}.
Indeed, we apply Theorem \ref{Main} to the semi-ample fibration
$\Phi:= \Phi_{|L^{\otimes m}|}: X \longrightarrow \mathbb{P}^{N_{m}}$ and 
$\omega :=\omega_{FS} $, where $\omega_{FS}$ is the Fubini-Study form.
Then $\Phi ^{*}\omega_{FS}+\ddbar \varphi $ is $q$-positive
for some function $\varphi$ by Theorem \ref{Main}.
The form represents the Chern class of $L^{\otimes m}$.
Therefore condition (B) implies (A).
Conversely, if $L$ is $q$-positive, 
it is cohomologically $q$-ample by the Andreotti-Grauert theorem.
Therefore the fibre dimension of the semi-ample fibration must be
at most $q$.

It remains to show the equivalence between condition (B) and (D).
In this step, we use the assumption that $X$ is projective. 
Assume that the fibre dimension of the semi-ample fibration is at most $q$.
Then for any subvariety $Z$ with $\dim{Z} > q$, 
we can take a curve on $Z$ which is not contracted by $\Phi$.
Then the degree of $L$ on the curve is positive by the projection formula.
Conversely, if there exists a fibre $F$ with $\dim{F} > q$, 
the degree on every curve in $F$ is zero.
\end{proof}

\section{Zariski-Fujita type theorem for big line bundles}
In this section, we prove Theorem \ref{non-ample}. 
Theorem \ref{non-ample} says that, 
a big line bundle is $q$-positive if and only if 
the restriction to the non-ample locus of the line bundle is $q$-positive.  
See \cite{ELMNP} or \cite[Section 3.5]{Bou04} for the definition and properties of the non-ample locus.
(Sometimes, the non-ample locus is called the augmented base locus
or the non-K\"ahler locus.)

\begin{theo}\label{non-ample}
Let $L$ be a big line bundle on a smooth projective
variety $X$. 
Assume that the restriction of $L$ to the non-ample locus $\mathbb{B}_{+}(L)$ 
is $q$-positive. 
Then $L$ is $q$-positive on $X$.

\end{theo}
Recall that 0-positive is positive 
in the usual sense (that is, ample).
Hence Theorem \ref{non-ample} claims that $L$ is ample on $X$
if the restriction to the non-ample locus of $L$ is ample. 
It can be seen as the parallel of the the Zariski-Fujita theorem
(see \cite{Zar89} and \cite{Fuj83} for the Zariski-Fujita theorem).

Throughout this section, we denote by $X$ a compact 
K\"ahler manifold and by $L$ a line bundle on $X$.
Moreover fix a smooth hermitian metric $h$ of $L$.
The Chern curvature $\sqrt{-1}\Theta_{h}(L)$ associated to $h$   
represents the first Chern class $c_{1}(L)$.

In this section, 
we treat a closed analytic set which may have the singularities on $X$.
For this purpose, we extend the $q$-positivity concept from a manifold to an analytic space.
Note the following definition
does not depend on the choice of a hermitian metric $h$ of $L$.

\begin{defi}\label{sing}
Let $V$ be a closed analytic set on $X$.
The restriction $L|_{V}$ to $V$ of $L$ is {\textit{q-positive}} 
if there exists a real-valued continuous function $\varphi $ on $V$ 
with the following condition:

For every point $p$ on $V$, there exist a neighborhood $U$ of $p$ on $X$ and 
a $C^{2}$-function $\widetilde {\varphi} $ on $U$
such that $\widetilde {\varphi} |_{V\cap U}= \varphi $ and 
the $(1,1)$-form 
$\sqrt{-1}\Theta_{h}(L) + \ddbar \widetilde {\varphi}$ has at least
$(n-q)$-positive eigenvalues on $U$.
\end{defi}

For the proof of Theorem \ref{non-ample}, 
we prepare the following lemma.

\begin{lemm}\label{supplement}
Let $V$ be a closed analytic set $($possibly not irreducible$)$ on $X$.
If the restriction $L|_{V}$ to $V$ of $L$ is $q$-positive, 
then $X$ allows a function $\varphi _{V} \in
C^{\infty}(X, \mathbb{R})$ on $X$
such that $\sqrt{-1}\Theta_{h}(L)  + \ddbar \varphi _{V}$ has 
at least $(n-q)$ positive eigenvalues 
on some neighborhood of $V$.
\end{lemm}

\begin{proof}

We take a smooth extension $\widetilde{\varphi }$ to $X$ 
of the potential function in Definition \ref{sing}.
Let $V= \bigcup _{i \in I} V_{i} $ be the irreducible decomposition.
From the construction of $\widetilde{\varphi }$, the restriction 
to ${(V_{i})}_{\reg}$ of 
$\sqrt{-1}\Theta_{h}(L)  + \ddbar \widetilde{\varphi }$ has 
at least $(\dim{V_{i}}-q)$ positive eigenvalues.
Then we can revise the positivity in the normal direction of $V$ 
 with the same argument as Proposition \ref{normal}.
It leads to Lemma \ref{supplement}.
\end{proof}
\vspace{0.2cm}\hspace{-0.6cm}
{\it Proof of Theorem \ref{non-ample}.}\hspace{0.1cm}
By the property of the non-ample locus 
(see \cite[Theorem 3.17]{Bou04}), 
there exists a $d$-closed current $T$ on $X$ with following properties:
\vspace{0.1cm}\\
\ \ \ \ (1)\ \ $T$ represents the first Chern class of $L$.
\\
\ \ \ \ (2)\ \ $T$ has analytic singularities along the non-ample locus
of $L$. 
\\
\ \ \ \ (3)\ \ For some hermitian form $\omega $ on $X$, 
the inequality $T \geq \omega $ holds as a $(1,1)$-current.
\vspace{0.1cm}\\
From property (1), we obtain an $L^{1}$-function 
$\varphi_{s}  $ on $X$ with 
$T = \sqrt{-1}\Theta_{h}(L) + \ddbar \varphi_{s} $.
On the other hand, by applying Lemma \ref{supplement} to the non-ample locus, 
we can obtain a function 
$\varphi _{\mathbb{B}_{+}} \in C^{\infty}(X, \mathbb{R})$ 
such that $\sqrt{-1}\Theta_{h}(L) + \ddbar \varphi _{\mathbb{B}_{+}}$ 
is $q$-positive
on some neighborhood $U$ on the non-ample locus.

Then we shall see that $\varphi_{s}$ and $ \varphi _{\mathbb{B}}$ can be glued.
For a real number $C>0$, we 
define the function $\psi_{C}$ to be  
$\psi_{C}:=\max \{ \varphi _{\mathbb{B}_{+}}-C,\varphi_{s} \}$.
For a large $C>0$, 
the function $\varphi _{\mathbb{B}}-C$ is smaller than $\varphi _{s}$
outside some neighborhood $U_{C}$ of the non-ample locus.
By taking a sufficiently large $C>0$, we may assume that 
$U_{C}$ is relatively compact in $U$. 

On the other hand, 
the function $\varphi_{s}$ has a polar set along the non-ample locus.
That is, $\varphi_{s} (x) = -\infty$ for any point $x$ on the non-ample 
locus.
Hence 
there exists a neighborhood $V_{C}$ of the non-ample locus 
such that $\varphi _{s}$ is smaller than $\varphi _{\mathbb{B}_{+}}-C$
even if $C$ is large.
We may assume $V_{C}$ is relatively compact in $U_{C}$.

Outside $U_{C}$, the $(1,1)$-form
\begin{equation*}
\sqrt{-1}\Theta_{h}(L) +\ddbar \psi_{C} = \sqrt{-1}\Theta_{h}(L)+ \ddbar \varphi _{s}
\end{equation*}
has $n$ positive eigenvalues.
On the other hand, inside $V_{C}$ the $(1,1)$-form
\begin{equation*}
\sqrt{-1}\Theta_{h}(L) +\ddbar \psi_{C} = \sqrt{-1}\Theta_{h}(L)+ \ddbar \varphi _{\mathbb{B}_{+}}
\end{equation*}
has $(n-q)$-positive eigenvalues.
In order to investigate the positivity on 
$\overline{U_{C}}\setminus V_{C}$, we prepare the following lemma.

\begin{lemm}\label{max}
Let $\gamma $ be a smooth $d$-closed $(1,1)$-form on $X$ 
and an $L^{1}$-function $\varphi_{i}$ $($for $i=1,2)$ 
with $\ddbar \varphi _{i}  \geq  \gamma$.  
Then the function $\max(\varphi_{1}, \varphi_{2})$ 
also satisfies 
$\ddbar \max(\varphi_{1}, \varphi_{2}) \geq \gamma$.
\end{lemm}
\begin{proof}
Notice that the conclusion is a local property.
We can locally take a smooth potential function 
of $\gamma $ since $\gamma $ is a $d$-closed form.
Thus we have $\gamma = \ddbar \psi$ 
for some function $\psi $.
By the assumption, $\ddbar( \varphi_{i} -\psi)$ is a positive 
current.
Therefore the Levi form of  
\begin{equation*}
\max( \varphi_{1}, \psi + \varphi_{2} - \psi )= 
 \max(\varphi_{1}, \varphi_{2}) - \psi 
\end{equation*}
is also a positive current.
It implies that $\ddbar 
\max(\varphi_{1}, \varphi_{2} )\geq \gamma$.

\end{proof}

Since $U_{C}$ is relatively compact in $U$, 
$\sqrt{-1}\Theta_{h}(L)+ \ddbar \varphi _{\mathbb{B}_{+}}$
is $q$-positive on $\overline{U_{C}}$. 
Certainly $\sqrt{-1}\Theta_{h}(L)+ \ddbar \varphi _{s}$ is $q$-positive 
($0$-positive) on $\overline{U_{C}} \setminus  V_{C}$. 
Therefore it follows from the lemma above that    
$\sqrt{-1}\Theta_{h}(L) +\ddbar \psi_{C}$ is $q$-positive 
on $\overline{U_{C}}\setminus V_{C}$.
The function $ \psi_{C}$ may not be smooth.
However we can approximate it with smooth functions
without the loss of the $q$-positivity since $\psi_{C}$ is 
continuous.
Therefore $L$ is $q$-positive on $X$.
\vspace{-10mm}
\\
\begin{flushright}
$\square$
\end{flushright}

When the dimension of the non-ample locus is smaller $q$, 
the assumption in Theorem \ref{non-ample} is automatically satisfied.
Thus we have the following corollary.
\begin{cor}\label{small-dim}
Assume the dimension of the non-ample locus of 
$L$ is less than or equal to $q$. 
Then $L$ is $q$-positive. 
\end{cor}
Under the assumption in Corollary \ref{small-dim}, 
$L$ is cohomologically $q$-ample 
(cf. \cite{Kur10}, \cite[Theorem 1.6]{Mat10}). 
Corollary \ref{small-dim} claims that 
the $q$-positivity has the same property.

\section{Relations with the holomorphic Morse inequality}
\label{hol}

In his paper \cite{Dem10}, Demailly proved the converse of the holomorphic Morse inequality on a surface. 
This result has the similarity to the converse of the Andreotti-Grauert theorem.
In this section, we explain the difference between his result and 
the result (Theorem \ref{AG1}) in this paper. 
First we recall the holomorphic Morse inequality which 
is closely related with the Andreotti-Grauert vanishing theorem.  
\begin{defi}
Let $L$ be a line bundle on a compact complex manifold $X$.
Then the asymptotic $q$-cohomology of $L$ is defined to be 
\begin{equation*}
\hat{h} ^{i} (L) : = \limsup_{m \to \infty} 
\frac{n!}{m^{n}} h^{i}(X, \mathcal{O}_{X}( L^{\otimes m} )).
\end{equation*}
\end{defi}

In his paper \cite{Dem85}, Demailly gave a relation between 
the dimension of the asymptotic  cohomology of a line bundle 
and certain Monge-Amp\`{e}re integrals of the curvature. 
It is so-called Demailly's holomorphic Morse inequality. 
For simplicity, we assume that $X$ is projective.
\begin{theo} $($\cite{Dem85}$)$.
For every holomorphic line bundle $L$ on a projective manifold $X$, 
we has the $($weak$)$ Morse inequality \\
\begin{equation*}
 \hat{h} ^{i} (L) \leq  \inf
 _{h\text{:hermitian metric on L}} \int_{X(h, i).} 
 (\sqrt{-1}\Theta_{h}(L))^{n} (-1)^{i}, 
\end{equation*}
where $h$ runs through smooth hermitian metrics on $L$, and  
${X(h, i)}$ is the set defined by 
\begin{equation*}
X(h, i) : = \big\{ x \in X \ |\ \sqrt{-1}\Theta_{h}(L)
\text{ has a signature } (n-i, i)\ \text{at}\ x.
 \big\}.
\end{equation*}
\end{theo}
The holomorphic Morse inequality would be seen as 
an asymptotic version of the Andreotti-Grauert vanishing theorem. 
In his paper \cite{Dem10-2}, Demailly conjectured that 
the inequality would actually be an equality. 
The conjecture has the similarity to Problem \ref{main}. 
In \cite{Dem10}, he showed the converse of holomorphic Morse inequality holds 
in the following case:\\
\ \ \ (1) The case where $X$ is projective surface. \\
\ \ \ (2) The case where $X$ is an arbitrary projective manifold and $i=0$. \\
Result (2) can be seen as a \lq\lq partial" converse of the Andreotti-Grauert theorem. 
However, Result (2) seems not to lead to Theorem \ref{AG1}.


\end{document}